%% file: tfc.tex
\documentclass{llncs}
\usepackage{amssymb,amsmath}
\usepackage{multicol}
\usepackage{graphicx}
\usepackage{epsfig}
\usepackage{amsfonts}
\usepackage{fullpage}
\newtheorem{assumption}{Assumption}

\newcommand{\ignore}[1]{}

\newtheorem{alphatheorem}{Theorem}

\title{A Brooks' Theorem for Triangle-Free Graphs}
\author{Mohammad Shoaib Jamall
\protect\footnotemark[1]\protect\footnotetext[1]{Research supported in part by a Reuben H. Fleet Foundation Fellowship, 
and an ARCS Foundation Fellowship.}}
\institute{Department of Mathematics, UC San Diego \\
\email{mjamall@math.ucsd.edu}}

\begin{document}

\maketitle

\input{abs}
\input{intro}
\input{algo}
\input{analysis}

\input{prelims}
\input{details}
\input{ack}

\bibliographystyle{amsplain}
\bibliography{color}

\end{document}

%% file: abs.tex
\begin{abstract}
Let $G$ be a triangle-free graph with maximum degree $\Delta(G)$.
We show that the chromatic number $\chi(G)$ is less than $67(1+o(1)) \Delta / \log \Delta$.
\end{abstract}

%% file: intro.tex
\section{Introduction}
A \emph{proper vertex coloring} of a graph is an assignment of colors to all vertices
such that adjacent vertices have distinct colors.
The \emph{chromatic number} $\chi(G)$ of a graph $G$ is the minimum number of colors
required for a proper vertex coloring. 
Finding the chromatic number of a graph is NP-Hard \cite{GJ79}. 
Approximating it to within a polynomial ratio is also hard \cite{K01}.
For general graphs, $\Delta(G) + 1$ is a trivial upper bound.
Brooks' Theorem \cite{B41} shows that $\chi(G)$ can be $\Delta(G) + 1$ only if 
$G$ has a component which is either a complete subgraph or an odd cycle.

A natural question is:
can this bound be improved for graphs without large complete subgraphs?
In $1968$, 
Vizing \cite{V68} had asked what the best possible upper bound for the 
chromatic number of a triangle-free graph was.
Borodin and Kostochka \cite{BK77}, Catalin \cite{C78}, and Lawrence \cite{L78}
independently made progress in this direction;
they showed that for a $K_4$-free graph,
$\chi(G) \leq 3(\Delta(G) + 2)/4$.
On the other hand, Kostochka and Masurova \cite{KM77}, 
and Bollob{\'a}s \cite{B78} separately showed that there are graphs of arbitrarily large 
\emph{girth}(length of a shortest cycle) with
$\chi(G)$ of order $\Delta(G)/\log \Delta(G)$.
%

In $1995$, Kim \cite{K95} proved that
\begin{equation*}
\chi(G) \leq (1+o(1))\frac{\Delta(G)}{\log \Delta(G)}
\end{equation*}
when $G$ has girth greater than $4$.
Later on, Johansson \cite{J96} showed that 
\begin{equation*}
\chi(G) \leq O(\frac{\Delta(G)}{\log \Delta(G)})
\end{equation*}
when $G$ is a triangle-free graph(girth greater than $3$).
Alon, Krivelevich and Sudakov \cite{AKS99}, and 
Vu \cite{V02} extended the method of Johansson to prove bounds on the chromatic number 
for graphs in which no subgraph on the set of all neighbors of a vertex has too many edges.

Both Kim and Johansson used the so-called \emph{semi-random method}
to show that the chromatic number of graphs with large girth is $O(\Delta(G) / \log \Delta(G))$.
This technique,
also known as the \emph{pseudo-random method}, 
or the \emph{R{\"o}dl nibble},
appeared first in Ajtai, Koml{\'o}s and Szemer{\'e}di \cite{AKS81}
and was applied to problems in hypergraph packings,
Ramsey theory, colorings, and list colorings \cite{FR85,K92,K96,K2_95,PS89}.
In general, given a set $S_1$, the goal is to show that there is an object in $S_1$ with a desired property $\mathcal{P}$.
This is done by locating a sequence of non-empty subsets 
$S_1 \supseteq \dots \supseteq S_{\tau}$ with $S_{\tau}$ having property $\mathcal{P}$.
A randomized algorithm is applied to $S_t$, which guarantees that
$S_{t+1}$ will be obtained with some non-zero(often small) probability.
For upper bounds on chromatic number, 
the semi-random method is used to prove the existence of a proper coloring with a limited number of colors.

In this paper we prove that the chromatic number of a triangle-free graph $G$ is less than $67(1+o(1)) \Delta / \log \Delta$.
As we will indicate in Section \ref{sec:pseudo}, our proof is derived from Kim's proof of an upper bound to the chromatic number of graphs with girth greater than $4$.
We believe our technique is simpler than Johanssons' which follows a different approach to that of Kims (see \cite{MR01} for a comparison of both).

We give our proof by analyzing an iterative algorithm for graph coloring.
To analyze this algorithm we identify a collection of random variables.
The expected changes to these random variables after a round of the algorithm are written in terms of the values of the random variables before the round.
We thus obtain a set of recurrence relations and prove that our random variables are concentrated around the solutions to the recurrence relations with some
positive probability.

We describe our algorithm in Section \ref{sec:pseudo}.
Section \ref{sec:sketch} contains motivation, which is followed by a formal description of the algorithm in Section \ref{sec:algorithm_details}.
We give an outline of the analysis in Section \ref{sec:outline}.
Section \ref{sec:prelim} contains some useful lemmas which we are used 
in Section \ref{sec:details} to give details of the analysis.

%% file: algo.tex
\section{An Iterative Algorithm for Coloring a Graph}
\label{sec:pseudo}
Our algorithm takes as input a triangle-free graph $G$ on $n$ vertices,
its maximum degree $\Delta$,
and the number of colors to use $\Delta / k$ where $k$ is a positive number.
It goes through rounds and assigns colors to more vertices each round.
Initially all vertices are \emph{uncolored}(no color assigned),
at the end we have a proper vertex coloring of $G$ with some probability.

\begin{definition} Let $t$ be a natural number. 
We define the following: \\
\begin{tabular}{ p{1.3cm} p{14cm} } 
$G_{t}$ & 
The graph induced on $G$ by the vertices that are uncolored at the beginning of round $t$. \\
$N_{t}(u)$ & 
The set of vertices adjacent to vertex $u$ in $G_t$. 
That is, the set of uncolored neighbors of $u$ at the beginning of round $t$. \\
$S_t(u)$ & The list of colors that may be assigned to vertex 
$u$ in round $t$, also called the \emph{palette} of $u$.
For all $u$ in $V(G)$,
\begin{equation*}
S_0(u) = \{1, \dots, \Delta/k\}.
\end{equation*} 
\\
$D_{t}(u,c)$ & The set of vertices adjacent to $u$ 
that may be assigned color $c$ in round $t$. That is,
\begin{equation*}
D_t(u,c) := \{v \in N_t(u) | c \in S_t(v)\}.
\end{equation*} 
\end{tabular}
\end{definition}
\begin{samepage}
It will be useful to define variables for the sizes of the sets $S_t(u)$ and $D_t(u,c)$.
\begin{definition} {\ } 
\begin{align*}
s_t(u) &= |S_t(u)| \\
d_t(u,c) &= |D_t(u,c)|
\end{align*}
\end{definition}
\end{samepage}
Observe that for every round $t$, vertex $u$ in $V(G)$, and color $c$ in $\{1, \dots, \Delta/k\}$,
\begin{equation*}
d_t(u,c) \leq \Delta, \;\; s_0(u) = \Delta/k, \;\; s_t(u) \leq \Delta/k.
\end{equation*}

\subsection{A Sketch of the Algorithm and the Ideas behind its Analysis}
\label{sec:sketch}
We say that a sequence $x(n)$ is $O(f(n))$ if there is a positive number $M$ 
such that $|x(n)| \leq M |f(n)|$.
All sequences in the big-oh are indexed by $\Delta$, the maximum degree of graph $G$.
Remember that each occurrence of the big-oh comes with a distinct constant $M$.
We start by considering an algorithm that colors $\Delta$-regular graphs with girth greater than $4$. 
The coloring produced is proper with positive probability. 

\begin{samepage}
Let $(d_t)$ and $(s_t)$ be sequences defined recursively as \\
\begin{center}
\begin{tabular}{l l p{1cm} l l}
$d_0$ 	& 	$:= \Delta$		&	&	$d_{t+1}$		& 	$:= d_t (1- c_1 \frac {s_t}{d_t}) c_2$ \\
$s_0$ 	& 	$:= \Delta/k$		&	&	$s_{t+1}$		& 	$:= s_t c_2$ \\
\end{tabular}
\end{center}
where $c_1$ and $c_2$ are constants between $0$ and $1$,
which are determined by the analysis of the algorithm.
\end{samepage}

{\ }
\begin{center}
\framebox{
\begin{minipage}{15cm}
{\bf Repeat at every round $t$, until $d_{t}/s_t < 1/2$}
\\
{\it Phase I - Coloring Attempt }
\\
For each vertex $u$ in $G_t$: \\
\hspace*{.5cm}
	{\it Awake} vertex $u$ with probabilitity $s_{t}/d_{t}$. \\
\hspace*{.5cm}
	If {\it awake}, assign to $u$ a color chosen from $S_t(u)$ uniformly at random. \\

{\it Phase II - Conflict Resolution}
\\
For each vertex $u$ in $G_t$: \\
\hspace*{.5cm}
	If {\it u is awake}, \\
\hspace*{1cm}
		uncolor $u$ if an adjacent vertex is assigned the same color. \\
\hspace*{.5cm}
	Remove from $S_t(u)$, all colors assigned to adjacent vertices. \\
\hspace*{.5cm}
	$S_{t+1}(u) = S_t(u).$ \\
{\bf end repeat} \\
\\
Permanently color each vertex $u$ in $V(G_t)$ with a color picked independently and uniformly at random from its palette $S_t(u)$.
\end{minipage}}
\end{center}
{\ } \\
Observe that $d_0 / s_0 = k$ and
\begin{equation*}
\frac{d_{t+1}}{s_{t+1}} = \frac {d_t}{s_t} - c_1.
\end{equation*}
In $O(k)$ rounds $d_t / s_t$ will be less than $1/2$, and this marks the end of the repeat-until block.

The algorithm above is derived from Kim \cite{K95}.
After some modifications, his analysis tells us that if graph $G$ has girth greater than $4$,
then there are constants $c_1$ and $c_2$ less than $1$ such that at each round $t$,
$\forall u \in V(G_t), \forall c \in S_t(u)$
\begin{equation}
\label{eq:inv_sk}
s_t(u) = s_t(1+o(1)), \;\;\; d_t(u,c) = d_t(1+o(1))
\end{equation}
with probability greater than $0$.
The equations above imply that after the repeat-until block,
if $\Delta$ is large enough, then with positive probability $s_t(u) > 2 d_t(u,c)$
for all uncolored vertices $u$ and colors $c$ in their palette.
Now, applying the result of Haxell \cite{H01}, 
we find that the random assignment of colors to all uncolored vertices in the final step of the algorithm
gives a properly colored graph with positive probability. 

\subsubsection{The problem with cycles of length $4$.}
The analysis for the above algorithm is probabilistic and proves the property in equation 
\eqref{eq:inv_sk} by induction,
showing concentration of the variables around their expectations.
It fails for graphs with $4$-cycles. An example illustrates why:
Consider a vertex $u$ whose $2$-neighborhood, 
the graph induced by vertices within distance $2$ of $u$,
is the complete bipartite graph 
$K_{\Delta,\Delta}$ with partitions $X$ and $Y$.
Suppose that $u$ and another vertex $v$ are in $X$.
If $v$ is colored with $c$ in round $0$ while $u$ remains uncolored,
then the set $D_{1}(u,c) = \emptyset$; this violates equation \eqref{eq:inv_sk} since $d_1 \geq 1$ 
if for example $k \geq 2$ and $\Delta \geq 2/c_2$.
So, when the graph has $4$-cycles, 
$d_{t+1}(u,c)$ is not necessarily concentrated around its expectation with positive probability, 
given the state of the algorithm at the beginning of round $t$.
We must modify the algorithm in two ways.
\begin{description}
\item[First Modification: A technique for coloring graphs with $4$-cycles.]
While $d_t(u,c)$ is not concentrated enough when the graph has $4$-cycles, 
our analysis will show that the average of 
$d_{t+1}(u,c)$ over all colors in the palette of a vertex $u$
is concentrated enough.
How does this benefit us?
Markov's famous inequality may be interpreted as:
a list of $s$ positive number which average $d$
has at most $s/q$ numbers larger than $qd$
for any positive number $q$.
We modify the algorithm so that at the end of each round $t$,
every vertex $u$ removes from its palette every color $c$ with
$d_{t+1}(u,c)$ larger than $2 d_{t+1}$.
Look at what happens in round $t=1$.
By a straightforward application of Markov's inequality, 
instead of equation \eqref{eq:inv_sk} we will have the less stringent property:
$\forall u \in V(G_t), \forall c \in S_t(u)$ \\
\begin{align}
\label{eq:inv2_sk}
s_t(u) \geq \frac {1}{2} s_t(1-o(1)), 
		\;\; d_t(u,c) \leq 2 d_t(1+o(1)).
\end{align}
with positive probability.
In fact, using a generalization of Markov's inequality,
the analysis will show that with a few more modifications our algorithm 
maintains, with positive probability, a slightly stronger property(still weaker than equation \eqref{eq:inv_sk}).

{\ }\;\;\;\; Equation \eqref{eq:inv_sk} implies that the $s_t(u)$ and $d_t(u)$
at all uncolored vertices $u$ are about the same.
It is a strong statement and helps in the proofs,
but is too much to maintain on graphs with $4$-cycles.
Equation \eqref{eq:inv2_sk} is weaker and is obtained by our algorithm with positive probability.
Moreover,
it is sufficient to ensure that after the repeat-until block, with positive probability, 
for each uncolored vertex $u$  and color $c$ in its palette, $s_t(u) \geq 2 d_t(u,c)$.
This is a key idea in our algorithm.
\item[Second Modification: Using independent random variables for easier analysis.]
Instead of waking up a vertex with some probability,
and then choosing a color from its palette uniformly at random;
for each uncolored vertex $u$ and color $c$ in its palette,
we will assign $c$ to $u$ independently with some probability.
In case multiple colors remain assigned to the vertex after the conflict resolution phase,
we will arbitrarily choose one of them to permanently color the vertex.
This modification, adapted from Johansson \cite{J96},
will make concentration of our random variables simpler.
\end{description}
Next we provide a formal description of the algorithm we have just motivated.

\subsection{A Formal Description of the Algorithm}
\label{sec:algorithm_details}
Let $(d_t)$ and $(s_t)$ be sequences defined recursively as
\begin{align*}
d_0 		&:= \Delta		&			
		&	d_{t+1}:= d_t (1- \frac{1}{16} e^{-1/2} \frac {s_t}{d_t}) e^{-1/2} \\
s_0	 	&:= \Delta/k	&
		&	s_{t+1} := s_t e^{-1/2}. \\
\end{align*}
\begin{equation}
\label{eq:rec_stage_1}
{\ }
\end{equation}
For round $t$, vertex $u$, and color $c$,
\begin{align}
\label{eq:pr_u_dr_c}
	\mathcal{F}_t(u,c) &:= \{c \mbox{ is not assigned to any vertex adjacent to }u\mbox{ in round }t\}
\end{align}
is an event in the probability space generated by the random choices of the algorithm in round $t$,
given the state of all data structures at the beginning of the round.

Let 
\begin{align*}
Desired\_\mathcal{F}_t &:= e^{-1/2}.
\end{align*}
\begin{center}
\framebox{
\begin{minipage}{15cm}

{\bf Repeat at every round $t$, until $d_{t} / s_t < 1 / 8$}\\
{\it Phase I - Coloring Attempt }
\\
For each vertex $u$ in $G_t$, and color $c$ in $S_t(u)$: \\
\hspace*{1cm}
	Assign $c$ to $u$ with probability $\frac 1{4} \frac 1 {d_t}.$ \\
\\
{\it Phase II - Conflict Resolution}
\\
For each vertex $u$ in $G_t$: \\
\hspace*{.5cm}
	{\it Phase II.1} \\
\hspace*{.5cm}
	Remove from $S_t(u)$, all colors assigned to adjacent vertices. \\
\hspace*{.5cm}
	{\it Phase II.2} \\
\hspace*{.5cm}
	For each color $c$ in $S_{t}(u)$, remove $c$ from $S_{t}(u)$ with probability
\begin{equation*}
1 - min(1, \frac{Desired\_\mathcal{F}_t(u,c)}{Pr(\mathcal{F}_t(u,c)}).
\end{equation*}
\hspace*{.5cm}
	If $S_t(u)$ has at least one color which is assigned to $u$, \\
\hspace*{1cm}
		then arbirarily pick an assigned color from $S_t(u)$ to permanently color $u$. \\
\\
{\it Phase III - Cleanup(discard all colors $c$ with $d_{t+1}(u,c) \gtrsim 2 d_{t+1}$ from palette)}
\\
For each vertex $u$ in $G_t$: \\
\hspace*{.5cm}
	$S_{t+1}(u) = S_t(u)$.\\
\hspace*{.5cm}
	Let $\alpha = 1 - |S_{t+1}(u)| / s_{t+1}.$ \\
\hspace*{.5cm}
	If $\alpha < 0$, then $\alpha = 0$, otherwise if $\alpha > 1/2$, then $\alpha = 1/2$. \\
\hspace*{.5cm}
	Let $\gamma$ be the smallest number in $[1, \infty)$ so that
\begin{equation*}
	Average_{c \in S_{t+1}(u)} d_{t+1}(u,c) \leq \frac {1 - 2 \alpha}{1 - \alpha} \gamma d_{t+1}.
\end{equation*}
\hspace*{.5cm}
	Remove all colors $c$ with $d_{t+1}(u,c) \geq 2 \gamma d_{t+1}$ from $S_{t+1}(u)$. \\
{\bf end repeat} \\
\\
Permanently color each vertex $u$ in $V(G_t)$ with a color picked independently and uniformly at random from its palette $S_t(u)$.
\end{minipage}}
\end{center}

%% file: analysis.tex
\section{The Main Theorem}
\label{sec:outline}


\begin{theorem}[Main Theorem]
Given $67 \Delta / \log \Delta$ colors and a triangle-free graph $G$ 
with maximum degree $\Delta$ large enough,
our algorithm finds a proper coloring of the graph with positive probability.
\end{theorem}
We need some lemmas to prove the Main Theorem and 
before that we need the following definition.
\begin{definition}
$d_{t}(v) = Average_{c \in S_{t}(v)} d_{t}(v,c)$\\
\end{definition}

\begin{lemma}[Main Lemma]
\label{lem:phase1}
Given $\psi > 1$ and a triangle-free graph $G$
with maximum degree 
$\Delta$, 
there is a positive constant $\beta$
such that for the sequence $(e_t)$ defined by
\begin{equation}
\label{eq:def_rece}
e_0 = 0, \mbox{\:\:\:} 
	e_{t+1} = 3 e_{t} + \beta(\sqrt{\frac{\psi}{s_{t}}}) \mbox{\;\; for } t > 0,
\end{equation}
if $s_t \gg \psi$ and $e_t \ll 1$ at round $t$, then
$$\forall u \in V(G_t), \exists \alpha \in [0, 1/2], \forall c \in S_{t}(u),$$
\begin{align*}
s_t(u) &\geq (1-\alpha) s_{t}(1 - e_{t}) \\
d_{t}(u) &\leq \frac{1-2 \alpha}{1- \alpha} d_{t}(1+ e_{t}) \\
d_{t}(u,c) &\leq 2 d_{t} (1+e_{t}) 
\end{align*}
with positive probability.
\end{lemma}
We will prove the Main Lemma in Section \ref{sec:details} and assume it in this section.
Using it we can immediately conclude the following.
\begin{corollary}
\label{cor:main_lemma}
Given the setup of the Main Lemma(Lemma \ref{lem:phase1}),
if $s_t \gg \psi$ and $e_t \ll 1$ at round $t$, then $\forall u \in V(G_t)$,
\begin{align*}
s_t(u) &\geq \frac 1 2 s_t(1-e_t) \\
d_{t}(u) &\leq d_{t}(1+ e_{t})
\end{align*}
with positive probability.
\end{corollary}

\begin{lemma}
\label{lem:t_stage_1}
The repeat-until block finishes in $16 e^{1/2} k$ rounds.
\end{lemma}
\begin{proof}
By the definition of sequences $(d_t)$ and $(s_t)$ in equation \eqref{eq:rec_stage_1}, we have
\begin{align*}
\frac{d_{t+1}}{s_{t+1}} &= \frac{d_{t}}{s_{t}} ( 1 - \frac {1} {16} \frac {s_{t}}{d_{t}} e^{-1/2}) \\
				&= \frac{d_{t}}{s_{t}} - \frac {1} {16} e^{-1/2}.
\end{align*}
Since $\frac {d_{0}}{s_{0}} = k$ we get $\frac {d_{t_1}}{s_{t_1}} \leq \frac 1 {4}$ after $16 e^{1/2} k$ rounds.
\qed
\end{proof}
Let
$$t_1 = 16 e^{1/2} k$$
be the last round of the repeat-until block. 
Then the following lemma is a straightforward application of equation \eqref{eq:rec_stage_1}.
\begin{lemma}
\label{lem:s_t1}
\begin{equation*}
s_{t_1} = \frac {\Delta} k exp(- 8 e^{1/2} k)
\mbox{   and,    if }
k \leq \frac 1 {9 e^{1/2}} \log \Delta
\mbox{ and }
\Delta
\mbox{ is large enough, then }  
s_{t_1} \gg 1.
\end{equation*}
\end{lemma}

\subsection{Bounding the Error Estimate in all Concentration Inequalities}
Now we look at $s_{t}$, which is used to bound the error term $e_{t}$.
\begin{lemma}
\label{lem:e_stage_1_2}
$$e_{t} \leq 3^{t} O(\sqrt{\frac {k \; exp(8 e^{1/2} k) \psi}{\Delta}}).$$
\end{lemma}
\begin{proof}
By Lemma \ref{lem:s_t1}, we have
$$s_{t_1} = \frac{\Delta} k exp(- 8 e^{1/2} k).$$
Note that in equation \eqref{eq:def_rece}, the recurrence for $e_{t}$,
the largest term is $O(\sqrt{\psi / s_t})$.
Since the sequence $(s_t)$ is decreasing,
we use Lemma \ref{lem:s_t1} to conclude that
$$O(\sqrt{\frac {\psi}{s_{t_1} }}) = 
O(\sqrt{\frac {k \; exp(8 e^{1/2} k) \psi}{\Delta }})$$ 
is the maximum this term can be. Thus we can simplify the recurrence for $e_{t}$ to
$$e_{t+1} = 3 e_{t} + \sqrt{\frac {k \; exp(8 e^{1/2} k) \psi}{\Delta }}).$$
Since $e_{0} = 0$, a simple upper bound for $e_t$ is given by
$$e_{t} \leq 3^{t} O(\sqrt{\frac {k \; exp(8 e^{1/2} k) \psi}{\Delta}})$$
where $\alpha$ is some positive constant.
\qed
\end{proof}

\begin{lemma}
\label{lem:pre_main}
Given $\Delta / k$ colors where $k \leq \frac 1 {67} (\log \Delta)$ and a triangle-free graph $G$ 
with maximum degree $\Delta$,
our algorithm reaches the end of the repeat-until block at round 
$t_1 = O(k)$ with $e_{t_1} \ll 1$, and
$\forall u \in V(G_{t_1}), c \in S_{t_1}(u)$
\begin{align*}
s_{t_1}(u) &\geq \frac {1} 2 s_{t_1}(1 - e_{t_1}) \\
d_{t_1}(u,c) &\leq 2 d_{t_1}(1 + e_{t_1})
\end{align*}
with positive probability.
\end{lemma}
\begin{proof}
Let $\psi = 3 \log \Delta$,
and let $t_1$ be the number of rounds to reach the end of the repeat-until block. 
Using Lemma \ref{lem:e_stage_1_2}, we get
$$ e_{t_1} \leq 3^{t_1} O(\sqrt{\frac{k \; exp(8 e^{1/2} k )\psi}{\Delta}}).$$
Using Lemma \ref{lem:s_t1} it is straightforward to show that
$$e_{t_1} \ll 1 \;\;\; \mbox{and} \;\;\; s_{t_1} \gg \psi$$
if $k \leq \frac 1 {67} \log \Delta$.
Applying Corollary \ref{cor:main_lemma} completes the proof.
\qed
\end{proof}
We may now prove the Main Theorem.
\begin{proof}[of the Main Theorem]
Using Lemma \ref{lem:pre_main}, we get
$$\forall u \in V(G_{t_1}), c \in S_{t_1}(u) \;\;\; s_{t_1}(u) \geq 2 d_{t_1}(u,c)$$
with positive probability.
Now Haxell \cite{H01} shows that the final step of our algorithm finds(randomly coloring all uncolored vertices)
finds a proper coloring with positive probability.
\qed
\end{proof}

%% file: prelims.tex
\section{Several Useful Inequalities}
\label{sec:prelim}
Now we look at some preliminaries which will be used in the proof details.
The next lemma describes what happens to the average value of a finite subset of
real numbers when large elements are removed. 
As shown in the statement of the lemma, 
it implies Markov's Inequality \cite{AS08}.
\begin{lemma}
\label{lem:mu1}
Consider a set of positive real numbers of size $n$ and average value $\mu$.
If we remove $\alpha n$ elements with value atleast $q \mu$ for some $q > 1$,
then the remaining points have average
$$ \mu' \leq \mu \frac{1-q \alpha}{1- \alpha}.$$
In particular, $\alpha \leq \frac 1 q$ since $\mu' \geq 0$.
\end{lemma}
\begin{proof}
The conclusion is obtained by a trivial manipulation of the following inequality
which relates $\mu$ and $\mu'$.
$$ q \mu \alpha + \mu' (1 - \alpha ) \leq \mu $$
\qed
\end{proof}
The next lemma describes what  happens when we add large elements to a finite
subset of real numbers.
\begin{lemma}
\label{lem:mu2}
Given the setup of Lemma \ref{lem:mu1}, 
if we add $\alpha n$ points with value $q \mu$ to the sample,
then the resulting larger sample has average
$$ \mu' = \mu \frac {1+q \alpha}{1+\alpha}$$
\end{lemma}
\begin{proof}
The conclusion is easily obtained from the following equation relating $\mu$ and $\mu'$.
$$ \mu'(1+\alpha) = \mu + q \mu \alpha $$
\qed
\end{proof}
We use the following lemma for computations with error factors.
\begin{lemma}
\label{lem:O}
Let $(A_n)$ be a sequence such that $0 < A_n < c < 1$(where $c$ is a constant),
and let $(e_n)$ be another sequence. Then
$$ 1 - A_n(1 + e_n) = (1 - A_n)(1+O(e_n)) $$
\end{lemma}
\begin{proof}
\begin{eqnarray*}
1 - A_n(1+e_n) &=& (1-A_n)(1+e_n) - e_n \\
&=& (1-A_n)(1+e_n) - (1-A_n) \frac {e_n}{(1-A_n)} \\
&=& (1-A_n)(1+e_n) - (1-A_n) O(e_n) \\
&=& (1-A_n)(1+O(e_n))
\end{eqnarray*}
\qed
\end{proof}
We use the following version of Azuma's inequality \cite{MR01}
to prove concentration of random variables.
\begin{alphatheorem}[Azuma's inequality]
\label{thm:azuma}
Let $X$ be a random variable determined by $n$ trials
$T_1, \dots, T_n$, such that for each $i$,
and any two possible sequences of outcomes
$t_1, \dots, t_i$ and $t_1, \dots, t_{i-1}, t_i'$:
$$|E[X | T_1=t_1, \dots, T_i = t_i] - E[X | T_1=t_1, \dots, T_i = t_i']| \leq \alpha_i$$
then
$$Pr(|X - E[X]| > t) \leq 2 e^{-t^2/(\sum \alpha_i^2)}$$
\end{alphatheorem}

We use the following version of the Lovasz Local Lemma \cite{MR01}
\begin{alphatheorem}[Lovasz Local Lemma]
\label{thm:lovasz}
Consider a set $\mathcal{E}$ of events such that for each $A \in \mathcal{E}$
\begin{itemize}
\item $Pr(A) \leq p < 1$, and
\item $A$ is mutually independent of a set of all but at most $d$ of the other events.
\end{itemize}
If $4pd \leq 1$, then with positive probability, none of the events in $\mathcal{E}$ occur.
\end{alphatheorem}

%% file: details.tex
\section{Proof of the Main Lemma}
\label{sec:details}
\label{sec:first_stage}
The following assumptions are repeatedly used in the lemmas of this section.
\begin{assumption}
\label{ass:indhyp}
Assume
$$s_t \gg \psi, e_t \ll 1$$
and with positive probability
$$\forall u \in V(G_t), \forall c \in S_{t}(u), \exists \alpha \in [0, 1/2],$$
\begin{align*}
s_t(u) &\geq (1-\alpha) s_{t}(1 - e_{t}) \\
d_{t}(u) &\leq \frac{1-2 \alpha}{1- \alpha} d_{t}(1+ e_{t}) \\
d_{t}(u,c) &\leq 2 d_{t} (1+e_{t}).
\end{align*}
\end{assumption}
All events in this section are in the probability space generated by our randomized algorithm in round $t+1$,
given the state of all data structures at the beginning of the round.
\begin{proof}[of the Main Lemma]
The proof is by induction on the round number $t$ using lemmas that follow.
The base case, when $t=0$, is trivially true. If we assume Assumption \ref{ass:indhyp}, the induction hypothesis,
for round $t$ then for each vertex $u$ in $V(G_{t+1})$, by Lemma \ref{lem:A} we have
\begin{align*}
Pr\{& d_{t+1}(u) \leq d_{t+1}(1 +O( e_{t} + \sqrt{\frac{\psi} {s_{t}}} + \frac{1}{d_{t}}) \}  \\
	& \;\;\;\;\;\;\;\;\;\;\;\; \geq 1 - e^{-\psi} O(1).
\end{align*}
For each $c$ in $S_{t+1}(u)$, by Lemma \ref{lem:S}, we have
\begin{align*}
Pr\{& \exists \alpha \in [0, \frac 1 2] \text{ such that } \\
	& s_{t+1}(u) \geq (1-\alpha) s_{t+1}(1 - 3 e_t + O(\sqrt{\frac{\psi}{s_t}} + \frac{1} {d_{t}})), \\
	& d_{t+1}(u) \leq \frac{1-2 \alpha}{1-\alpha} d_{t+1}(1+ 3 e_t + O(\sqrt{\frac{\psi}{s_t}} + \frac{1} {d_{t}})), \\
	& d_{t+1}(u,c) \geq (1-\alpha) 2d_{t+1}(1 - 3 e_t + O(\sqrt{\frac{\psi}{s_t}} + \frac{1} {d_{t}})) \} \\
	& \;\;\;\;\;\;\;\;\;\;\;\; \geq 1 - e^{-\psi} O(1).
\end{align*}
Each of the events in the probabilities above is dependent on at most $O(\Delta^2)$ other such events.
If $\psi \geq 3 \log \Delta$ and $\Delta$ is large enough, 
then we use Theorem \ref{thm:lovasz} to conclude that Assumption \ref{ass:indhyp} holds for round $t+1$ .
\qed
\end{proof}
The above proof of the Main Lemma required Lemmas 
\ref{lem:A} and \ref{lem:S}.
The rest of this section will prove these lemmas. 
Next we consider the state of the palettes just before the cleanup phase of round $t$.

\begin{definition}
Let ${\tilde S}_t(u)$ be the list of colors in the palette of vertex $u$
in round $t$ just before the cleanup phase,
and let ${\tilde s}_t(u)$ be the size of ${\tilde S}_t(u)$. 
That is, ${\tilde S}_t(u)$ is obtained from $S_t(u)$
by removing colors discarded in the conflict resolution phase.
\end{definition}

\begin{lemma}
\label{lem:sg}
Given Assumption \ref{ass:indhyp}, for each vertex $u$ in $V(G_{t+1})$ we have
\begin{align*}
Pr\{& s_{t}(u) e^{-1/2}(1 - \frac 1 2 e_t - O(\sqrt{\frac{\psi}{s_{t}}})) \leq \tilde s_{t}(u) \leq s_{t}(u) e^{-1/2}(1 + O(\sqrt{\frac{\psi}{s_{t}}}))\}  \\
	& \;\;\;\;\;\;\;\;\;\;\;\; \geq 1 - e^{-\psi} O(1).
\end{align*}
\end{lemma}

\begin{proof}
Suppose $u$ is an uncolored vertex at the beginning of round $t$, 
and $c$ a color in its palette.
\begin{align*}
\lefteqn{Pr\{\text{$c$ is removed from $S_{t}(u)$ in phase II.1}\}}  \;\;\;\;\;\; \\
	&= 1 - Pr\{\text{no neighbor of $u$ is assigned $c$}\} \\
	&= 1 - \prod_{v \in D_{t}(u,c)}(1 - Pr\{ \text{$v$ is assigned $c$} \}) \\
	&= 1 - \prod_{v \in D_{u,c}}(1 - \frac 1 {4} \frac 1 {d_{t}} ) \\
	&\leq 1 - (1 - \frac 1 {4} \frac 1 {d_{t}})^{d_{t}(u,c)} \\
	&\leq 1 - (1- \frac 1 {4} \frac 1 {d_{t}})^{2 d_{t}(1+e_{t})} \\
	&\leq 1 - e^{\log (1 - \frac 1 {4} \frac 1 {d_{t}})2 d_{t}(1+e_{t})} 
		& \langle \log(1+x) = x + O(x^2) \rangle \\
	&\leq 1 - e^{(-\frac 1 {4} \frac 1 {d_{t}} + O(\frac 1 {d_{t}})^2) 2 d_{t}(1+e_{t})} 
		& \langle \mbox{Assumption \ref{ass:indhyp}} \rangle \\
	&\leq 1 - e^{-1/2}(1-\frac 1 2 e_t +O(\frac{1}{d_{t}}))
\end{align*}
In phase II.2 of round $t+1$ we remove colors from the palette using an appropriate bernoulli variable, to get
$$Pr\{c \notin \tilde S_{t}(u)\} = 1 - e^{-1/2}(1 - \frac 1 2 e_t + O(\frac {1}{d_{t}})).$$
Using linearity of expectation
$$ E[\tilde s_{t}(u)] = s_{t}(u) e^{-1/2}(1 - \frac 1 2 e_t + O(\frac{1}{d_{t}})).$$

For concentration of $\tilde s_t(u)$,
suppose $s_t(u) = m$.
Let $c_1, \dots, c_m$ be the colors in $S_t(u)$.
Then $\tilde S_t(u)$ may be considered a random variable determined by $m$
trials $T_1, \dots, T_m$ where $T_i$ is the set of vertices in $G_t$
that are assigned color $c_i$ in round $t$. Observe that $T_i$
affects $\tilde S _t(u)$ by at most 1 given $T_1, \dots, T_{i-1}$.
Now using Theorem \ref{thm:azuma} we get,
$$Pr\{|\tilde s_t(u) - E[\tilde s_t(u)]| \geq \sqrt{\psi s_t(u)}\} \leq e^{-\psi} O(1).$$
\qed
\end{proof}
We now focus on the sets $D_t(u,c)$. The following two lemmas will help.
\begin{lemma}
\label{lem:u_cl_c}
Let $u$ be an uncolored vertex, and $c$ be a color in its palette at the beginning of round $t$.
Then given Assumption \ref{ass:indhyp}, we have
$$Pr\{\text{$u$ is assigned  $c$ and $c \in \tilde S_t(u)$}\} = 
	\frac 1 {4} \frac 1 {d_{t}} e^{-1/2}(1- \frac 1 2 e_t + O(\frac {1}{d_{t}})).$$
\end{lemma}
\begin{proof}
\begin{align*}
	\lefteqn{Pr\{ \text{$u$ is assigned  $c$ and $c \in \tilde S_t(u)$} \}} \;\;\;\;\;\; \\ 
		&= Pr\{ \text{$u$ is assigned $c$}\} Pr\{ c \in \tilde S_t(u)\} \\
		&= \frac 1 {4} \frac 1 {d_{t}} e^{-1/2}(1- \frac 1 2 e_t + O(\frac{1}{d_{t}}))
			& \langle \mbox{Equation \eqref{eq:u_dr_c_naturally}} \rangle
\end{align*}
\qed
\end{proof}
The following lemma is a consequence of the previous one.
\begin{lemma}
\label{lem:u_cl}
Let $u$ be an uncolored vertex at the beginning of round $t$. 
Then given Assumption \ref{ass:indhyp}, we have
$$Pr\{u\text{ is colored}\} 
	\geq \frac {1}{16} \frac {s_t} {d_{t}} e^{-1/2}(1 - 3 e_t + O(\frac {1}{d_{t}})).$$
\end{lemma}
\begin{proof}
Consider the event
$$ \{u \mbox{ is colored}\} = 
	\bigcup_{c \in S_t(u)} \{u \mbox{ is assigned } c \mbox{ and } c \in \tilde S_t(u)\}.$$
Since the events in the union on the right hand side of the equation above are independent,
$$Pr\{u \mbox{ is colored}\} = 1 - 
	\prod_{c \in S_t(u)}
		(1 - Pr\{u \mbox{ is assigned } c \mbox{ and } c \in \tilde S_t(u)\}).$$
Now using Lemma \ref{lem:u_cl_c}, we get
\begin{align*}
\lefteqn{Pr \{u \mbox{ is colored}\}} \;\;\;\;\;\; \\ 
	&\geq 1 - (1 - \frac 1 {4} \frac 1 {d_t} e^{-1/2}(1 - \frac 1 2 e_t + O( \frac 1 {d_t})))^{s_t(u)}\\
	&\geq 1 - (1 - \frac 1 {4} \frac 1 {d_t} e^{-1/2}
		(1 - \frac 1 2 e_t + O(\frac 1 {d_t})))^{\frac {1}2 s_t(1-e_t)} 
		& \langle \mbox{Assumption \ref{ass:indhyp}} \rangle \\
	&\geq 1 - exp(- \frac {1} {8} \frac {s_t} {d_t} e^{-1/2}(1 - \frac 3 2 e_t + O(\frac 1 {d_t}))) \\
	&\geq 1 - (1 - \frac {1} {16} \frac {s_t} {d_t} e^{-1/2}(1 - \frac 3 2 e_t + O( \frac 1 {d_t})))  \\
	&=  \frac {1} {16} \frac {s_t} {d_t} e^{-1/2}(1 - \frac 3 2 e_t + O( \frac 1 {d_t})).
\end{align*}
\qed
\end{proof}
\begin{samepage}
We will need the following definitions.
\begin{definition} {\ } \\
\begin{itemize}
\item
Let $\tilde D_t(u,c)$ be the set of uncolored vertices that have color $c$
in their palettes  and are uncolored in round $t$, just before the cleanup phase. That is,
$$\tilde D_t(u,c) = D_t(u,c) \setminus (\{v | c \notin \tilde S _t(v)\} \cup \{v | v \mbox{ is colored in round } t\}).$$
\item
Let $\tilde d_t(u,c)$ be the size of $\tilde D_t(u,c)$.
\item
$ \bar d_{t}(u) := \sum_{c \in \tilde S_{t}(u)} \tilde d_{t}(u,c) = \sum_{c \in S_{t}(u)} 1_{\{c \in \tilde S_{t}(u)\}} \tilde d_{t}(u,c) $
\item
$ \tilde d_{t}(u) := \frac{\bar d_{t}(u)}{\tilde s_{t}(u)}$
\end{itemize}
\end{definition}
\end{samepage}

\begin{lemma}
\label{lem:tilde_a}
Given Assumption \ref{ass:indhyp}, for each vertex $u$ in $V(G_{t+1})$ we have
\begin{align*}
Pr\{& \tilde d_{t}(u) 
\leq d_{t}(u) (1 - \frac {1}{16} \frac {s_{t}}{d_{t}} e^{-1/2})e^{-1/2}(1 + 2 e_t + O(\sqrt{\frac{\psi} {s_{t}}} + \frac{1}{d_{t}} + \sqrt{\frac{\psi d_{t}}{s_{t} d_{t}(u)}}))\}  \\
	& \;\;\;\;\;\;\;\;\;\;\;\; \geq 1 - e^{-\psi} O(1).
\end{align*}
\end{lemma}
\begin{proof}
Let $u$ be an uncolored vertex at the beginning of round $t$,
and let $c$ be a color in its palette.
For a vertex $v$ in $D_{t}(u,c)$, 
Lemma \ref{lem:u_cl_c} implies that $Pr\{v$ is colored with $d\} = O(1 / {d_t})$
for any color $d$ in $S_t(v)$.
Thus,
\begin{eqnarray*}
Pr(\{c \notin \tilde S_t(v) \} \cap \{v \mbox{ is colored}\}) 
	&=& \sum_{d \in S_t(v)} Pr(\{c \notin \tilde S_t(v) \} \cap \{v \mbox{ is colored with } d\}) \\
	&=& \sum_{d \in S_t(v)} Pr\{c \notin \tilde S_t(v) | v \mbox{ is colored with } d \} 
			Pr\{v \mbox{ is colored with } d\} \\
	&=& Pr\{c \notin \tilde S_t(v)\}(1 + O(\frac 1 {d_t})) \sum_{d \in S_t(v)}  
			Pr\{v \mbox{ is colored with } d\} \\
	&=& Pr\{c \notin \tilde S_t(v)\}Pr\{v \mbox{ is colored}\}(1 + O(\frac 1 {d_t})).
\end{eqnarray*}
A straightforward computation now shows that
\begin{equation}
\label{eq:v_dr_c_n_cln}
Pr(\{c \notin \tilde S_t(v) \} \cap \{v \mbox{ is not colored}\}) = \\
	Pr\{c \notin \tilde S_t(v)\}Pr\{v \mbox{ is not colored}\}(1 + O(\frac 1 {d_t})).
\end{equation}
Now, $v$ is removed from the set $D_t(u,c)$ if either it is colored or color $c$
is removed from its palette. This means that event
\begin{eqnarray*}
\{v \notin \tilde D_t(u,c) \} = \{\text{$v$ is colored}\} 
	\cup (\{c \notin \tilde S_t(v)\} \cap \{ \text{$v$ is not colored}\}).
\end{eqnarray*}
Since $G$ is triangle-free, $u$ and $v$ do not have any common neighbors.
This implies that
\begin{align*}
\lefteqn{Pr\{v \notin \tilde D_t(u,c) | c \in \tilde S_t(u)\}} \;\;\;\;\;\; \\
	&= Pr\{v \notin \tilde D_t(u,c)\}(1+O(\frac 1 {d_t})) \\
	&= (Pr\{\text{$v$ is colored}\}
		+ Pr(\{c \notin \tilde S_t(v) \} \cap \{ \text{$v$ is not colored}\}))(1+O(\frac 1 {d_t})) \\
	&= (Pr\{\text{$v$ is colored}\} 
		+ Pr\{c \notin \tilde S_t(v) \} Pr\{ \text{$v$ is not colored}\}) (1+O(\frac 1 {d_t})) 
		& \langle \mbox{equation \eqref{eq:v_dr_c_n_cln}}\rangle \\
	&= (Pr\{\text{$v$ is colored}\} 
		+ (1 - e^{-1/2})(1 - Pr\{\text{$v$ is colored}\}))(1 + O(\frac {1}{d_{t}})) 
		& \langle \mbox{equation \eqref{eq:u_dr_c_naturally}}\rangle \\		
	&= (1 - (1 - Pr\{\text{$v$ is colored}\}) e^{-1/2})(1 + O(\frac {1}{d_{t}})) \\
	&\geq (1 - (1 - \frac {1} {16} \frac {s_{t}} {d_{t}} e^{-1/2}) e^{-1/2})
											(1+ 2 e_t + O(\frac {1}{d_{t}}))
		& \langle \mbox{Lemma \ref{lem:u_cl}}\rangle.
\end{align*}
Using linearity of expectation
\begin{equation}
E[\tilde d_{t}(u,c) | c \in \tilde S_{t}(u)] 
= E[\tilde d_{t}(u,c)](1+O(\frac 1{d_t})) 
\leq d_{t}(u,c)(1 - \frac {1}{16} \frac {s_{t}} {d_{t}} e^{-1/2}) e^{-1/2}(1+ 2 e_t + O(\frac {1}{d_{t}})).
\end{equation}
Now using the above bound
\begin{eqnarray*}
E[\bar d_{t}(u)] &=& \sum_{c \in S_{t}(u)} Pr\{c \in \tilde S_{t}(u)\} E[\tilde d_{t}(u,c) | c \in \tilde S_{t}(u)] \\
	&\leq& e^{-1/2} \sum_{c \in S_{t}(u)} d_{t}(u,c)(1 - \frac {1}{16} \frac {s_{t}}{d_{t}} e^{-1/2}) e^{-1/2}(1+ 2 e_t + O(\frac {1}{d_{t}})) \\
	&\leq& e^{-1/2} s_{t}(u) d_{t}(u)(1 - \frac {1}{16} \frac {s_{t}}{d_{t}} e^{-1/2}) e^{-1/2}(1+ 2 e_t + O(\frac {1}{d_{t}}))
\end{eqnarray*}

For concentration of $\bar d_t(u)$,
suppose $s_t(u) = m$.
Let $c_1, \dots, c_m$ be the colors in $S_t(u)$.
Then $\bar d_t(u)$ may be considered a random variable determined by the random trials
$T_1, \dots, T_m$, where $T_i$ is the set of vertices in $G_t$ that are assigned color $c_i$
in round $t$.
Observe that $T_i$ affects $\bar d_t(u)$ by at most $d_t(u,c)$.

Thus $\sum \alpha_i^2$ in the statement of Theorem \ref{thm:azuma}
is less that $\sum_{c \in S_t(u)} d_t^2(u,c)$.
This upperbound is maximized when the $d_{t}(u,c)$ take the extreme values of
$2 d_{t}$ and $0$ subject to $d_{t}(u) = \frac 1 {s_{t}(u)} \sum_{c \in S_{t}(u)} d_{t}(u,c)$.
Thus
$$\sum \alpha_i^2 \leq O( (d_{t})^2 d_{t}(u) s_{t}(u) / d_{t}) \leq O(s_{t}(u) d_{t} d_{t}(u))$$
Using Theorem \ref{thm:azuma}, we get
\begin{align*}
Pr\{& \bar d_{t}(u) 
- e^{-1/2} s_t(u) d_{t}(u)(1 - \frac {1} {16} \frac {s_{t}} {d_{t}} e^{-1/2}) e^{-1/2}(1 + 2 e_t + O(\frac {1}{d_{t}})) \geq O(\sqrt{\psi s_t(u)d_{t} d_{t}(u)}) \} \\
	&\;\;\;\;\;\;\;\;\;\;\;\;  \leq e^{-\psi} O(1).
 \end{align*}
Lemma \ref{lem:sg} says that 
\begin{align*}
Pr\{& s_{t}(u) e^{-1/2}(1 - \frac 1 2 + O(\sqrt{\frac{\psi}{s_{t}}})) \leq \tilde s_{t}(u) \leq s_{t}(u) e^{-1/2}(1 + O(\sqrt{\frac{\psi}{s_{t}}}))\}  \\
	& \;\;\;\;\;\;\;\;\;\;\;\; \geq 1 - e^{-\psi} O(1).
 \end{align*}
Combining the above two inequalities we have
\begin{align*}
Pr\{& \frac{\bar d_{t}(u)}{\tilde s_{t}(u)} 
- d_{t}(u)(1 - \frac {1} {16} \frac {s_{t}} {d_{t}} e^{-1/2}) e^{-1/2}(1 + 2 e_t + O(\frac {1}{d_{t}} + \sqrt{\frac{\psi}{s_{t}}})) \geq O(\sqrt{\frac{\psi d_{t} d_{t}(u)}{s_{t}}}) \} \\
	& \;\;\;\;\;\;\;\;\;\;\;\; \leq e^{-\psi} O(1).
 \end{align*}
Therefore
$$Pr\{\frac{\bar d_{t}(u)}{\tilde s_{t}(u)} 
\geq d_{t}(u)(1 - \frac {1} {16}\frac {s_{t}}{d_{t}} e^{-1/2}) e^{-1/2}
(1 + 2 e_t + O(\sqrt{\frac{\psi} {s_{t}}} + \frac{1}{d_{t}} + \sqrt{\frac{\psi d_{t}}{s_{t} d_{t}(u)}}))\} 
\leq e^{-\psi} O(1).$$
\qed
\end{proof}

Note that $\frac{\bar d_{t}(u)}{\tilde s_{t}(u)}$ is the average $|\tilde D_t(u,c)|$ 
at a vertex $u$ at the end phase II.
Phase III only brings this average down by removing colors with large $d_{u,c}$.
Thus we get the next lemma almost immediately.

\begin{lemma}
\label{lem:A}
Given Assumption \ref{ass:indhyp}, for each $u$ in $V(G_{t+1})$ we have
\begin{align*}
Pr\{& d_{t+1}(u) \leq d_{t+1}(1 +2 e_t + O( \sqrt{\frac{\psi} {s_{t}}} + \frac{1}{d_{t}}) \}  \\
	& \;\;\;\;\;\;\;\;\;\;\;\; \geq 1 - e^{-\psi} O(1).
\end{align*}
\end{lemma}
\begin{proof}
Let $u$ be a vertex in $V(G_{t+1})$. By Lemma \ref{lem:tilde_a}
\begin{align*}
Pr\{& \tilde d_{t}(u) 
\leq d_{t}(u) (1 - \frac {1}{16} \frac {s_{t}}{d_{t}} e^{-1/2})e^{-1/2}(1 + 2 e_t + O(\sqrt{\frac{\psi} {s_{t}}} + \frac{1}{d_{t}} + \sqrt{\frac{\psi d_{t}}{s_{t} d_{t}(u)}}))\}  \\
	& \;\;\;\;\;\;\;\;\;\;\;\; \geq 1 - e^{-\psi} O(1).
\end{align*}
Now
\begin{align*}
\lefteqn{d_{t}(u)(1 - \frac {1} {16} \frac {s_{t}}{d_{t}} e^{-1/2}) e^{-1/2}
 (1 + 2 e_t + O(\sqrt{\frac{\psi} {s_{t}}} + \frac{1}{d_{t}} + \sqrt{\frac{\psi d_{t}}{s_{t} d_{t}(u)}})) }
 	\;\;\;\;\;\; \\
&= d_{t}(u)(1 - \frac {1} {16} \frac {s_{t}}{d_{t}} e^{-1/2}) e^{-1/2}
(1 + 2 e_t + O(\sqrt{\frac{\psi} {s_{t}}} + \frac{1}{d_{t}})) + O(\sqrt{\frac{\psi d_{t}(u) d_{t}}{s_{t}}}) \\
&= d_{t}(u)(1 - \frac {1} {16} \frac {s_{t}}{d_{t}} e^{-1/2}) e^{-1/2}
(1 + 2 e_t + O(\sqrt{\frac{\psi} {s_{t}}} + \frac{1}{d_{t}})) + d_{t} O(\sqrt{\frac{\psi d_{t}(u)}{s_{t} d_{t}}}) \\
&\leq d_{t}(1 - \frac {1} {16} \frac {s_{t}}{d_{t}} e^{-1/2}) e^{-1/2}
(1 + 2 e_t + O(\sqrt{\frac{\psi} {s_{t}}} + \frac{1}{d_{t}})) + d_{t} O(\sqrt{\frac{\psi}{s_{t}}}) \\
&= d_{t}(1 - \frac {1} {16} \frac {s_{t}}{d_{t}} e^{-1/2}) e^{-1/2}
(1 + 2 e_t + O(\sqrt{\frac{\psi} {s_{t}}} + \frac{1}{d_{t}})).
\end{align*}
Thus the event
\begin{align*}
\lefteqn{\{\frac{\bar d_{t}(u)}{\tilde s_{t}(u)} 
\geq d_{t}(1 - \frac {1} {16} \frac {s_{t}}{d_{t}} e^{-1/2}) e^{-1/2}
(1 + 2 e_t + O(\sqrt{\frac{\psi} {s_{t}}} + \frac{1}{d_{t}}))\} }
 	\;\;\;\;\;\; \\
&\subseteq \{\frac{\bar d_{t}(u)}{\tilde s_{t}(u)} 
\geq d_{t}(u)(1 - \frac {1} {16} \frac {s_{t}}{d_{t}} e^{-1/2}) e^{-1/2}
(1 + 2 e_t + O(\sqrt{\frac{\psi} {s_{t}}} + \frac{1}{d_{t}} + \sqrt{\frac{\psi d_{t}}{s_{t} d_{t}(u)}}))\}.
\end{align*}
Therefore
\begin{align*}
\lefteqn{e^{-\psi} O(1) }
 	\;\;\;\;\;\; \\
&\leq Pr\{\frac{\bar d_{t}(u)}{\tilde s_{t}(u)} 
\geq d_{t}(1 - \frac {1} {16} \frac {s_{t}}{d_{t}} e^{-1/2}) e^{-1/2}
(1 + 2 e_t + O(\sqrt{\frac{\psi} {s_{t}}} + \frac{1}{d_{t}}))\} \\
&\leq Pr\{\frac{\bar d_{t}(u)}{\tilde s_{t}(u)} 
\geq d_{t}(u)(1 - \frac {1} {16} \frac {s_{t}}{d_{t}} e^{-1/2}) e^{-1/2}
(1 + 2 e_t + O(\sqrt{\frac{\psi} {s_{t}}} + \frac{1}{d_{t}} + \sqrt{\frac{\psi d_{t}}{s_{t} d_{t}(u)}}))\}.
\end{align*}
\qed
\end{proof}

Next we show that in the cleanup phase of round $t$, 
a vertex discards so many colors that its palette size in round $t+1$ 
becomes less than
$\frac {1}2 s_{t+1}(1-e_{t+1})$ with a very small probability.
\begin{lemma}
\label{lem:S}
Given Assumption \ref{ass:indhyp}, for each vertex $u$ in $V(G_{t+1})$ we have
\begin{align*}
Pr\{& \exists \alpha \in [0, \frac 1 2] \mbox{ such that } \forall c \in S_{t+1}(u)\\
	& s_{t+1}(u) \geq (1-\alpha) s_{t+1}(1 - \frac 3 2 e_t + O(\sqrt{\frac{\psi}{s_t}} + \frac{1} {d_{t}})), \\
	& d_{t+1}(u) \leq \frac{1-2 \alpha}{1-\alpha} d_{t+1}(1+ \frac 3 2 e_t + O(\sqrt{\frac{\psi}{s_t}} + \frac{1} {d_{t}})), \\
	& d_{t+1}(u,c) \leq 2 d_{t+1}(1+ \frac 3 2 e_t + O(\sqrt{\frac{\psi}{s_t}} + \frac{1} {d_{t}})) \} \\
	& \;\;\;\;\;\;\;\;\;\;\;\; \geq 1 - e^{-\psi} O(1).
\end{align*}
\end{lemma}
\begin{proof}

Consider vertex $u \in V(G_{t+1})$. 
Using Assumption \ref{ass:indhyp}, at round $t$, $\exists \alpha \in [0, \frac 1 2]$ such that $ s_{t}(u) \geq (1-\alpha) s_{t}(1 - e_{t}) $
and $d_{t}(u) \leq \frac {1-2\alpha}{1-\alpha} d_{t}(1 + e_{t})$. 
By Lemma \ref{lem:sg} we get
\begin{align*}
Pr\{& s_{t}(u) e^{-1/2}(1 - \frac 1 2 e_t + O(\sqrt{\frac{\psi}{s_{t}}})) \leq \tilde s_{t}(u) \leq s_{t}(u) e^{-1/2}(1 + \frac 1 2 e_t + O(\sqrt{\frac{\psi}{s_{t}}}))\}  \\
	& \;\;\;\;\;\;\;\;\;\;\;\; \geq 1 - e^{-\psi} O(1).
\end{align*}
Now
\begin{eqnarray*}
\tilde s_{t}(u) &=& s_{t}(u) e^{-1/2}(1+ \frac 1 2 e_t + O(\sqrt{\frac{\psi}{s_t}} + \frac{1} {d_{t}})) \\
	&\geq& (1-\alpha)s_t e^{-1/2}(1+ \frac 1 2 e_t + O(\sqrt{\frac{\psi}{s_t}} + \frac{1} {d_{t}})) \\
	&\geq& (1-\alpha) s_{t+1}(1+ \frac 1 2 e_t + O(\sqrt{\frac{\psi}{s_t}} + \frac{1} {d_{t}})).
\end{eqnarray*}
By Lemma \ref{lem:tilde_a} we get
\begin{align*}
Pr\{& \tilde d_{t}(u) 
\leq d_{t}(u) (1 - \frac {1}{16} \frac {s_{t}}{d_{t}} e^{-1/2})e^{-1/2}(1 + \frac 3 2 e_t + O(\sqrt{\frac{\psi} {s_{t}}} + \frac{1}{d_{t}} + \sqrt{\frac{\psi d_{t}}{s_{t} d_{t}(u)}}))\}  \\
	& \;\;\;\;\;\;\;\;\;\;\;\; \geq 1 - e^{-\psi} O(1).
\end{align*}
Now
\begin{align*}
\tilde d_{t}(u) &\leq d_{t}(u)(1-\frac {1} {16} \frac {s_{t}}{d_{t}} e^{-1/2})e^{-1/2} 
(1+  \frac 3 2 e_t + O(\sqrt{\frac {\psi}{s_{t}}} + \frac{1} {d_{t}})) \\
	&\leq \frac {1-2\alpha}{1-\alpha} d_{t}(1-\frac {1} {16} \frac {s_{t}}{d_{t}} e^{-1/2})e^{-1/2} 
(1+  \frac 3 2 e_t + O(\sqrt{\frac {\psi}{s_{t}}} + \frac{1} {d_{t}})) \\
	&\leq \frac {1-2\alpha}{1-\alpha} \gamma d_{t+1}.
\end{align*}
where $\gamma$ is the smallest number in $[1,\infty)$ for which the above inequality is true.
Combining the preceding inequalities, we get
$$\gamma = 1+  3 e_t + O(\sqrt{\frac{\psi}{s_{t}}} + \frac{1} {d_{t}}).$$

In the cleanup phase of our algorithm(given in Section \ref{sec:algorithm_details}), 
the change in palette is equivalent to the following process.
\begin{enumerate}
\item
Add $\frac{\alpha}{1-\alpha} \tilde s_t(u)$ arbitrary colors to 
$u$'s palette, with $\tilde d_{t}(u,c) = 2 \gamma d_{t+1}$.
This adjusts the palette size to $ \tilde s_{t}(u) \geq s_{t+1}(1+ 3 e_t + O(\sqrt{\psi/s_{t}} + 1/ d_t ))$.
Lemma \ref{lem:mu2} ensures that the adjusted new average is 
$ \tilde d_{t}(u) \leq \gamma d_{t+1}$
\item
Remove all the colors with $d_{t}(u,c) \geq 2 \gamma d_{t+1}$. 
\end{enumerate}
Now we use Lemma \ref{lem:mu1}, setting $\mu$ to $\gamma d_{t+1}$ and $q\mu$ to $2 \gamma d_{t+1}$, to get
$$s_{t+1}(u) \geq (1-\alpha) s_{t+1}(1 + 3 e_t + O(\sqrt{\frac{\psi}{s_t}} + \frac{1} {d_{t}}))$$
and 
$$d_{t+1}(u) \leq \frac{1-2 \alpha}{1-\alpha} d_{t+1}(1+ 3 e_t + O(\sqrt{\frac{\psi}{s_t}} + \frac{1} {d_{t}})).$$
The result is obtained using Lemmas \ref{lem:sg} and \ref{lem:tilde_a}.
\qed
\end{proof}

%% file: ack.tex
\section{Acknowledgement}
I am indebted to Fan Chung  her comments and support.